\documentclass[12pt]{amsart}
\usepackage{amsmath, amsthm, amscd, amsfonts}

\newtheorem{theorem}{Theorem}[section]
\newtheorem{lemma}[theorem]{Lemma}
\newtheorem{proposition}[theorem]{Proposition}
\newtheorem{corollary}[theorem]{Corollary}
\theoremstyle{definition}
\newtheorem{definition}[theorem]{Definition}
\newtheorem{example}[theorem]{Example}

\theoremstyle{remark}

\newtheorem{remark}[theorem]{Remark}
\numberwithin{equation}{section}

\newfont{\kh}{msbm10}

\begin{document}
\title[The Gap between unbounded regular operators]
{The Gap between unbounded Regular operators}
\author{K. Sharifi}
\address{Kamran Sharifi, \newline Department of Mathematics,
Shahrood University of Technology, P. O. Box 3619995161-316,
Shahrood, Iran} \email{sharifi.kamran@gmail.com and
sharifi@shahroodut.ac.ir}


\begin{abstract}
We study and compare the gap and the Riesz topologies of the space
of all unbounded regular operators on Hilbert C*-modules. We show
 that the space of all bounded adjointable operators on Hilbert
C*-modules is an open dense subset of the space of all unbounded
regular operators with respect to the gap topology. The
restriction of the gap topology on the space of all bounded
adjointable operators is equivalent with the topology which is
generated by the usual operator norm. The space of regular
selfadjoint Fredholm operators on Hilbert C*-modules over the
C*-algebra of compact operators is path-connected with respect to
the gap topology, however, the result may not be true for some
Hilbert C*-modules.
\end{abstract}
\maketitle

\section{Introduction.}

Hilbert C*-modules are essentially objects like Hilbert spaces,
except that inner product, instead of being complex-valued, takes
its values in a C*-algebra. The theory of these modules, together
with bounded and unbounded operators, is not only rich and
attractive in its own right but forms an infrastructure for some
of the most important research topics in operator algebras. They
play an important role in the modern theory of C*-algebra, in
KK-theory, in noncommutative geometry and in quantum groups.

A (left) {\it pre-Hilbert C*-module} over a C*-algebra
$\mathcal{A}$ is a left $\mathcal{A}$-module $E$ equipped with an
$\mathcal{A}$-valued inner product $\langle \cdot , \cdot \rangle
: E \times E \to \mathcal{A}\,, \ (x,y) \mapsto \langle x,y
\rangle$, which is $\mathcal A$-linear in the first variable $x$
and has the properties:
$$ \langle x,y \rangle=\langle y,x \rangle ^{*}, \ \langle ax,y \rangle
=\langle x,y \rangle \, \ {\rm for} \ {\rm all} \ a \ {\rm in} \
\mathcal{A},$$ $$ \langle x,x \rangle \geq 0 \ \ {\rm with} \
   {\rm equality} \ {\rm only} \ {\rm when} \  x=0.$$

A pre-Hilbert $\mathcal{A}$-module $E$ is called a {\it Hilbert $
\mathcal{A}$-module} if $E$ is a Banach space with respect to the
norm $\| x \|=\|\langle x,x\rangle \| ^{1/2}$. If $E$, $F$ are
two Hilbert $ \mathcal{A}$-modules then the set of all ordered
pairs of elements $E \oplus F$ from $E$ and $F$ is a Hilbert
$\mathcal{A}$-module with respect to the $\mathcal A$-valued
inner product $\langle (x_{1},y_{1}),(x_{2},y_{2})\rangle=
\langle x_{1},x_{2}\rangle_{E}+\langle y_{1},y_{2}\rangle _{F}$.
It is called the direct {\it orthogonal sum of $E$ and $F$}. A
Hilbert $\mathcal A$-submodule $E$ of a Hilbert $\mathcal
A$-module $F$ is an orthogonal summand if $E \oplus E^\bot = F$,
where $E^\bot$ denotes the orthogonal complement of $E$ in $F$.
The papers \cite{FR2, FR3, SCH}, some chapters in \cite{GVF, WEG},
and the books by E. C. Lance \cite{LAN} are used as standard
sources of reference.

As a convention, throughout the present paper we assume
$\mathcal{A}$ to be an arbitrary C*-algebra (i.e. not necessarily
unital), we also assume $\mathcal{K}(H)$ to be the C*-algebra of
all compact operators on an arbitrary Hilbert space $H$. Since we
deal with bounded and unbounded operators at the same time we
simply denote bounded operators by capital letters and unbounded
operators by small letters. We use the notations $Dom(.)$,
$Ker(.)$ and $Ran(.)$ for domain, kernel and range of operators,
respectively.

Suppose $E,\ F$ are Hilbert $\mathcal{A}$-modules. We denote the
set of all $\mathcal{A}$-linear maps $T: E \to F$ for which there
is a map $T^*: F \to E$  such that the equality $\langle Tx,y
\rangle = \langle x,T^*y \rangle$ holds for any $ x \in E,\ y \in
F$ by $B(E,F)$. The operator $T^*$ is called the {\it adjoint
operator} of $T$. $B(E,E)$ is denoted by $B(E)$.

Unbounded regular operators were first introduced by Baaj and
Julg in \cite{B-J} and later they were studied more by Woronowicz
and Napi\'{o}rkowski in \cite{W-N, WOR}. Lance gave a brief
indication in his book \cite{LAN} about unbounded regular
operators on Hilbert C*-modules. An operator $t$ from a Hilbert
$\mathcal{A}$-module $E$ to another Hilbert $\mathcal{A}$-module
$F$ is said to be {\it regular} if
\newcounter{cou001}
\begin{list}{(\roman{cou001})}{\usecounter{cou001}}
\item $t$ is closed and densely defined,
\item its adjoint $t^*$ is also densely defined, and
\item the range of $1+t^*t$ is dense in $E$.
\end{list}

Note that if we set $\mathcal{A}= \mathbb{C}$ i.e. if we take $E$
to be a Hilbert space, then this is exactly the definition of a
densely defined closed operator, except that in that case, both
the second and the third condition follow from the first one. In
the frame work of Hilbert C*-modules, one needs to add these
extra conditions in order to get a reasonably good theory. The
reader is encouraged to study the publications \cite{F-S, FS2, KK,
PAL1} for more detailed information about unbounded operators on
Hilbert C*-modules.

The {\it gap topology} is induced by the metric $d(t,s)=
\|P_{G(t)}-P_{G(s)} \|$ where $P_{G(t)}$ and $P_{G(s)}$ are
projections onto the graphs of densely defined closed operators
$t, s$, respectively. The gap topology on the space of all densely
defined closed operators has been studied systematically in the
book \cite{KAT} and in the seminal paper by Cordes and Labarousse
in \cite{C-L}. Recently the gap topology on the space of all
unbounded selfadjoint Fredholm operators has been reconsidered in
\cite{BLP, JOA, NIC}.

We study the gap topology of the space of all unbounded regular
operators on arbitrary Hilbert C*-modules. We also introduce an
strictly stronger topology than the gap topology of the space of
all unbounded regular operators, which is called Riesz topology.
We show that the space of all bounded adjointable operators on
Hilbert C*-modules is an open dense subset of the space of all
unbounded regular operators with respect to the gap topology.
Moreover the restriction of the gap topology on the space of all
bounded adjointable operators is equivalent with the topology
which is generated by the usual operator norm. The gap metric
will help us to find an isometric operation preserving map of the
space of all densely defined closed operators on Hilbert
C*-modules over the C*-algebra of compact operators onto the
space of all densely defined closed operators on a suitable
Hilbert space. This fact together with a result of Cordes and
Labarousse \cite{C-L} give us the opportunity to characterize the
path-connected components of the space of unbounded regular
Fredholm operators on Hilbert $\mathcal{K}(H)$-modules with
respect to the gap topology. Indeed, every two unbounded regular
Fredholm operators are homotopic if and only if they have the
same index.

\section{Preliminaries}

In this section we would like to recall some definitions and
simple facts about regular operators on Hilbert
$\mathcal{A}$-modules. For details see chapters 9 and 10 of
\cite{LAN}, and the papers \cite{F-S, KK, W-N, WOR}. Then we will
introduce and compare the Riesz and gap topologies of the space of
all unbounded regular operators.

Let $E, F$  be Hilbert $\mathcal{A}$-modules, we will use the
notation $t : Dom(t) \subseteq E \to F$ to indicate that $t$ is
an $\mathcal{A}$-linear operator whose domain $Dom(t)$ is a dense
submodule of $E$ (not necessarily identical with $E$) and whose
range is in $F$. A densely defined operator $t: Dom(t) \subseteq
E \to F$ is called {\it closed} if its graph $G(t)=\{(x,tx): \ x
\in Dom(t)\}$ is a closed submodule of the Hilbert
$\mathcal{A}$-module $E \oplus F$. If $t$ is closable, the
operator $s : Dom(s) \subseteq E \to F$ with the property $G(s)=
\overline{G(t)}$ is called the {\it closure} of $t$ denoted by
$s= \overline{t}$. A densely defined operator $t: Dom(t)
\subseteq E \to F$ is called {\it adjointable} if it possesses a
densely defined map $t^*: Dom(t^*) \subseteq F \to E$ with the
domain
$$
   Dom(t^*)= \{y \in F : {\rm there} \ {\rm exists} \ z \in E \
   {\rm such} \ {\rm that} \ \langle tx,y \rangle _{F} = \langle
   x , z \rangle _{E} \ {\rm for} \ {\rm any} \ x \in Dom(t) \}
$$
which satisfies the property $\langle tx,y \rangle_{F} = \langle
x,t^*y \rangle_{E}$,  for any $x \in Dom(t), \ y \in Dom(t^*)$.
This property implies that $t^*$ is a closed $\mathcal{A}$-linear
map. A densely defined closed $\mathcal{A}$-linear map $t: Dom(t)
\subseteq E \to F$ is called {\it regular} if it is adjointable
and the operator $1+t^*t$ has a dense range. We denote the set of
all regular operators from $E$ to $F$ by $R(E,F)$. $R(E,E)$ is
denoted by $R(E)$. A densely defined operator $t$ is regular if
and only if its graph is orthogonally complemented $E \oplus F$
(cf.~ \cite[Corollary 3.2]{F-S}). If $t$ is regular then $t^*$ is
regular and $t=t ^{**}$, moreover $t^*t$ is regular and
selfadjoint. Define $Q_{t}=(1+t^*t)^{-1/2}$,
$R_{t}=(1+t^*t)^{-1}=Q_{t}^{2}$ and $F_{t}=tQ_{t}$, then
$Ran(Q_{t})=Dom(t)$,  $0 \leq Q_{t} \leq 1$ in $B(E,E)$ and
$F_{t}\in B(E,F)$ (cf.~\cite{LAN}, chapter 9). The bounded
operator $F_{t}$ is called the bounded transform (or Woronowicz
transform) of the regular operator $t$. The map $t\to F_{t}$
defines a bijection
$$
R(E,F) \to \{ T \in B(E,F):\| T \|\leq 1 \ \, {\rm and} \ \
Ran(1- T^* T ) \ {\rm is} \ {\rm dense} \ {\rm in} \ F \},
$$
(cf. \cite[Theorem 10.4]{LAN}). This map is adjoint-preserving,
i.e. $F_{t}^*=F_{t^*}$, and for the bounded transform
$F_{t}=tQ_{t}=t(1+t^*t)^{-1/2}$ we have $\|F_{t}\|\leq 1$ and $
t=F_{t}Q_{t}^{-1}=F_{t}(1-F_{t}^*F_{t})^{-1/2}. $ A regular
operator $t \in R(E)$ is called selfadjoint if $t^*=t$. Obviously
a regular operator $t$ is selfadjoint if and only if its bounded
transform $F_t$ is selfadjoint.
\begin{corollary}Let $T \in R(E,F)$ be a regular operator
and $F_{T}$ be its bounded transform. Then $T \in B(E,F)$ if and
only if  $ \| F_{T} \|< 1$.
\end{corollary}
\begin{proof}For each $T \in R(E,F)$ the bounded adjointable operator
$Q_{T}:E \rightarrow Ran(Q_T)=Dom(T)\subseteq E$ is invertible
and satisfies $Q_{T}^2=1-F_{T}^*F_{T}$. Therefore $T \in B(E,F)$
if and only if $Dom(T)=E$, if and only if $ \| F_{T} \|< 1$.
\end{proof}

There is a natural metric on the set of regular operators, the
so-called gap metric. Let $t \in R(E,F)$ then $E \oplus F=G(t)
\oplus V(G(t^*))$, where $V \in B(E \oplus F,F \oplus E )$ is
defined by $V(x,y)=(y,-x)$. The orthogonal projection $P_{G(t)}:E
\oplus F \rightarrow E \oplus F$ can be described through the
following matrix
\begin{equation} \label{projection}
P_{G(t)}=\begin{pmatrix} R_{t}& t^*R_{t^*}\\ tR_{t} &
1-R_{t^*}\end{pmatrix} \in B(E \oplus F).
\end{equation}
It follows from \cite[(9.7)]{LAN} and the equalities $F_t
F_{t}^*=1-R_t^*$ and $(F_t Q_t)^*=(tR_{t})^*=t^*R_{t^*}$.
\begin{definition}Let $t,s \in R(E,F)$ then the gap metric on the
space of all unbounded regular operators is defined by $
d(t,s)=\|P_{G(t)}-P_{G(s)} \|$ where $P_{G(t)}$ and $P_{G(s)}$
are orthogonal projections onto $G(t)$ and $G(s)$, respectively.
The topology induced by this metric is called gap topology.
\end{definition}

Let $E$, $F$ be two Hilbert $\mathcal{A}$-modules and operators
$t,s$ be in $R(E,F)$. An equivalent picture of the gap metric is
now definable by using (\ref{projection}) as well as the fact
that $(tR_{t})^*=t^*R_{t^*}$. Indeed, the following metric, which
is again denoted by $d$, is uniformly equivalent to the gap metric
\begin{equation} \label{gapone}
d(t,s)= {\rm sup} \{ \, \|R_{t}-R_{s} \|,\|R_{t^*}-R_{s^*}
\|,\|tR_{t}-sR_{s} \|  \, \}.
\end{equation}

\begin{remark}Let $t \in R(E,F)$ be a regular operator and $F_{t}$
be its bounded transform. For every bounded adjointable operator
$S$ in $B(E,F)$ of norm $\| S \| \leq 1$ the operator $
\mathcal{F}(S):=1-S^*S$ is a positive operator. we also have
\begin{align*}
R_{t}~&=Q_{t}^2=1-F_{t}^*F_{t}=\mathcal{F}(F_{t})\,,\\
R_{t^*}&=Q_{t^*}^2=1-F_{t}F_{t}^*=\mathcal{F}(F_{t}^*)\,,\\
tR_{t}&=tQ_{t}^2=F_{t}Q_{t}=F_{t}\mathcal{F}(F_{t})^{1/2}.
\end{align*}
Therefore we can reformulate the gap metric (\ref{gapone}) via
the bounded transforms of regular operators $t$ and $s$ as
follows:
\begin{equation} \label{gaptwo}
d(t,s)= {\rm sup} \{ \, \|
\mathcal{F}(F_{t})-\mathcal{F}(F_{s})\|,
\|\mathcal{F}(F_{t}^*)-\mathcal{F}(F_{s}^*)\|,
\|F_{t}\mathcal{F}(F_{t})^{1/2}-F_{s}\mathcal{F}(F_{s})^{1/2}\| \,
\}.
\end{equation}
\end{remark}

The Riesz topology of unbounded selfadjoint operators on Hilbert
spaces has been investigated in \cite{BLP, KAU, NIC}. Their works
motivate us for the following definition.

\begin{definition}Let $t,s \in R(E,F)$ then the Riesz metric on the
space of all unbounded regular operators is defined by $ \sigma
(t,s)=\|F_{t}-F_{s} \|$. The topology induced by this metric is
called Riesz topology.
\end{definition}

\begin{lemma}The Riesz topology is stronger than the gap topology
on $R(E,F)$.
\end{lemma}

\begin{proof}Let $\{t_{n}\}$ be a sequence in $R(E,F)$ that is
convergent to a regular operator $t$ with respect to the Riesz
topology, i.e. $ \sigma(t_{n},t)=\|F_{t_n}-F_{t} \| \rightarrow
0$. By elementary methods and continuity of the function
$g(x)=\sqrt{x}$ on $[0,+\infty)$ we can deduce
\begin{align*}
&\|\mathcal{F}(F_{t_n})-\mathcal{F}(F_{t})\|~=\|F_{t_n}^*F_{t_n}-F_{t}^*F_{t}\|
\rightarrow 0 \,,\\
&\|\mathcal{F}(F_{t_n}^*)-\mathcal{F}(F_{t}^*)\|=\|F_{t_n}F_{t_n}^*-F_{t}F_{t}^*\|
\rightarrow 0 \,,\\
&\|\mathcal{F}(F_{t_n})^{1/2}-\mathcal{F}(F_{t})^{1/2}\|\rightarrow 0 \,,\\
&\|F_{t_n}\mathcal{F}(F_{t_n})^{1/2}-F_{t}\mathcal{F}(F_{t})^{1/2}\|\rightarrow
0 \,.
\end{align*}
Therefore (\ref{gaptwo}) implies that
$$d(t_{n},t)= {\rm sup} \{ \, \|
\mathcal{F}(F_{t_{n}})-\mathcal{F}(F_{t})\|,
\|\mathcal{F}(F_{t_{n}}^*)-\mathcal{F}(F_{t}^*)\|,
\|F_{t_{n}}\mathcal{F}(F_{t_{n}})^{1/2}-F_{t}\mathcal{F}(F_{t})^{1/2}\|
\, \} \rightarrow 0,$$ i.e. the sequence $\{t_{n}\}$ is gap
convergent to the regular operator $t$.
\end{proof}

\begin{corollary} Let
$ \mathcal{B}=\{ T \in B(E,F):\| T \|\leq 1 \ \, {\rm and} \ \
Ran(1- T^* T ) \ {\rm is} \ {\rm dense} \ {\rm in} \ F \}$, then
the map $$( \mathcal{B}, \|.\|) \to (R(E,F),d)\, , \ F_t \mapsto
t= F_{t}(1-F_{t}^*F_{t})^{-1/2}$$ is bijective and continuous.
\end{corollary}
Bijectivity and continuity of the map are obtained from Theorem
10.4 of \cite{LAN} and Lemma 2.5.
Suppose $E$ is a Hilbert $ \mathcal{A}$-module, $UB(E)$ and
$SR(E)$ denote unitary elements of $B(E)$ and selfadjoint
elements of $R(E)$.

\begin{remark} Let $E$ be a Hilbert $ \mathcal{A}$-module and let $ t \in
SR(E)$. According to Lemmata 9.7, 9.8 of \cite{LAN}, the
operators $t \pm i$ are bijection (see also \cite[Proposition
6]{KK}). Then

\begin{equation*} \begin{split}
c_{t}: ~SR(E) & \longrightarrow \mathcal{C}= \{ U \in UB(E)~:~ 1-U
~ \text{has
dense range } \}, \\
 ~t \,   & \, \, \mapsto c_{t}=(t-i)(t+i)^{-1}.
\end{split}
\end{equation*}
is a bijection which is called the {\it Cayley transform} of $t$,
cf. \cite[Theorem 10.5]{LAN}. The Cayley transform $c_t$ can be
written as $c_t=1-2i(t+i)^{-1}$. Thus
$(t+i)^{-1}-(s+i)^{-1}=\frac{i}{2}(c_{t}-c_{s})$, for each $t,s
\in SR(E)$.
\end{remark}
\begin{corollary}\label{selfadjoint-gap}On the space $SR(E)$ the gap metric is uniformly
equivalent to the metric $ \widetilde{d}$ given by
\begin{equation} \label{gap-Cayly}
\widetilde{d}(t,s):= \|(t+i)^{-1}-(s+i)^{-1} \|
=\frac{1}{2}\|\,c_{t}-c_{s}\|, \  {\rm for~all~} t,s \in SR(E).
\end{equation}
\end{corollary}
\begin{proof}For each $t,s \in SR(E)$, the expression (\ref{gapone}) can be
written as follows:
\begin{equation} \label{gapthr}
d(t,s)= {\rm sup} \{ \, \|R_{t}-R_{s} \|,\|tR_{t}-sR_{s} \| \, \}.
\end{equation}
On the other hand the operators $t \pm i$ are bijective, hence the
identities  $(t-i)^{-1}=(t+i)(t^2+1)^{-1}=tR_t+iR_t \,, \
(t+i)^{-1}=(t-i)(t^2+1)^{-1}=tR_t-iR_t$ hold, which yield
\begin{equation}\label{gapfor}\begin{split}
  R_t&=\frac{1}{2i}\bigl((t-i)^{-1}-(t+i)^{-1}\bigr),\\
  tR_t&=\frac 12\bigl((t-i)^{-1}+(t+i)^{-1}\bigr).
\end{split}
\end{equation}
Now from (\ref{gap-Cayly}), (\ref{gapthr}) and (\ref{gapfor}) we
infer that $ \frac{1}{2} \widetilde{d}(t,s) \leq d(t,s) \leq
\widetilde{d}(t,s)$, for all $t,s \in SR(E)$.
\end{proof}

The following example attributed to Fuglede is used to show that
the Riesz and gap metrics are different, cf. \cite{BLP, NIC}.

\begin{example}\label{selfadjoint-gap.EX}let $\mathcal{A}$ be unital
C*-algebra and $H_{ \mathcal{A}}$ be the standard Hilbert $
\mathcal{A}$-module which is countably generated by orthonormal
basis $ \xi _{j}=(0,...,0,1,0,...,0),~ j \in \mathbb{N}$. For
every integer $n \geq 0$ we define $ t_{n}:Dom(t_{n})= \{ \, \sum
\lambda_{j} \xi_{j}: \  \sum \, j^2 \, |\lambda _{j}|^{2} < +
\infty \} \subseteq H_{ \mathcal{A}} \longrightarrow  H_{
\mathcal{A}}$ by
\[ \qquad t_{n}( \xi_{j})=
\begin{cases}
~j \xi_{j}~ & \text{ if   $j \neq n$,}\\
-j \xi_{j}~ & \text{ if $j= n$}\ .
\end{cases}
\]
The sequence $t_{n}$ of selfadjoint regular operator converges to
the selfadjoint regular operator $t_0$ in gap topology. To see
this, we apply (\ref{gap-Cayly}) and get
$$\widetilde{d}(t_n,t_0)=\|(t_n+i)^{-1}-(t_0+i)^{-1} \|=
\|(t_n+i)^{-1}\xi_{n}-(t_0+i)^{-1} \xi_{n} \|=|\frac{1}{i-n}-
\frac{1}{i+n}| \rightarrow 0,~ \text{as}~ n \to \infty .$$  But
$$ \sigma (t_n,t_0) \geq \|F_{t_n} \xi_{n} -F_{t_0} \xi_{n} \|=
\frac{2n}{ \sqrt{1+n^2} } \rightarrow 2.$$ In  view of Corollary
\ref{selfadjoint-gap} this shows that the Riesz topology is
strictly stronger than the gap topology.
\end{example}
\section{On the gap topology}
Recall that every bounded adjointable operator is regular, that
is, for Hilbert C*-modules $E,~ F$, the space $B(E,F)$ can be
regarded as a subset of $R(E,F)$. We show that the space of all
bounded adjointable operators on Hilbert C*-modules is an open
dense subset of the space of all unbounded regular operators with
respect to the gap topology. Then we can conclude that the space
of odd bounded adjointable operators is a dense subset of odd
unbounded regular operators. The author believe that these
results are new even in the case of Hilbert spaces.

\begin{lemma} \label{gap-open} Let $E,F$ be Hilbert C*-modules.
Then the space $B(E,F)$ is open in $R(E,F)$ with respect to the
gap metric $d$.
\end{lemma}
\begin{proof}Let $S \in B(E,F)$ then $ \| F_{S} \|^2<1$  by
Corollary 2.1, and so, there is a real number $ \delta $ such that
$0 < \delta < 1-\| F_{S} \|^2<1$. We claim $ \{T \in
R(E,F)~:~d(T,S) < \delta \} \subseteq B(E,F)$. Let $T$ be a
(possibly unbounded) operator in  $R(E,F)$ and $d(T,S) < \delta$,
then
$$  \|F_{T}^*F_{T}\| - \|F_{S}^*F_{S}\| \leq
\|F_{T}^*F_{T}-F_{S}^*F_{S}\|=\|\mathcal{F}(F_{T})-\mathcal{F}(F_{S})\|
\leq d(T,S) < \delta .$$ That is, $ \|F_{T}  \|^2=\| F_{T}^*F_{T}
\|< \delta+\|F_{S}  \|^2 <1$, and $T$ is therefore bounded by
Corollary 2.1.
\end{proof}
\begin{proposition}The metric which is given by the usual norm of
bounded operator and the gap metric $d$ are equivalent on the
space of all bounded adjointable operators.
\end{proposition}

A similar result has been proved in the case of Hilbert spaces in
\cite[Addendum]{C-L}. Our argument seems to be shorter.
\begin{proof}Let $T,S \in B(E,F)$, we use the expression
(\ref{gapone}) to show that there exist real numbers $M_1,M_2$
such that $M_2 \|T-S \| \leq d(T,S) \leq M_1 \|T-S \|$. Since $
\|R_{T}  \| ,  \|R_{S}  \| \leq 1$, we have
\begin{eqnarray*}
\|R_T-R_S \|& \leq& \|R_T \| \, \|S^*S-T^*T \| \, \|R_S \| \\
            & \leq& \|S^*(S-T)+(S^*-T^*)T \| \\
            & \leq& \|S^* \| \, \|S-T \|+\| S^*-T^* \| \, \|T \| \\
            & = & ( \|T \|+\|S \|) \, \|T-S \|,\\
\|TR_T-SR_S \|& = &  \|(T-S)R_T+S(R_T-R_S) \| \\
            & \leq &\|T-S \|+\|S \| \, \|R_T-R_S \| \\
            & \leq &\|T-S \| \, (1+\| S \| \,(\| S \|+ \|T \|))
            \,.
\end{eqnarray*}
Similarly $\|R_{T^*}-R_{S^*} \| \leq ( \|T \|+\|S \|) \, \|T-S
\|$. Therefore the above inequalities imply that $d(T,S) \leq M_1
\|T-S \|$, where $M_1={\rm max} \, \{ \, \|T \|+\|S \| ~,~1+\| S
\| \,(\| S \|+ \|T \|) \}$. We know that $T-S=(T-S)R_{T}\,
R_{T}^{-1}$ and $(T-S)R_{T}=TR_T-SR_S+S(R_S-R_T)$, so we obtain
\begin{eqnarray*}
   \|T-S \| & = & \|(T-S)R_T \| \, \|R_{T}^{-1}\| \\
            & \leq &(1+\|T\|^{2})\|(T-S)R_T \| \\
            & \leq &(1+\|T\|^{2}) \, (\|TR_T-SR_S\|+ \|S \| \|R_S-R_T \|)\\
            & \leq &(1+\|T\|^{2}) \, (d(T,S)+ \|S \| \, d(T,S))\\
            & \leq &(1+\|T\|^{2}) \, (1+ \|S \|)\, d(T,S).
\end{eqnarray*}
Therefore $M_2 \|T-S \| \leq d(T,S)$, for $M_2=[(1+\|T
\|^2)(1+\|S\|)]^{-1}$.
\end{proof}
\begin{remark}Let $t \in R(E,F)$ and $F_{t}$ be its bounded
transform. Then $1-F_{t}^*F_{t}$ has dense range if and only if
$c1-F_{t}^*F_{t}$ has dense range for each real number $c\geq 1$
(cf. \cite[Lemma 10.1 and Corollary 10.2]{LAN}).
\end{remark}
\begin{theorem}\label{gap-dense}Let $E,F$ be Hilbert C*-modules, then $B(E,F)$
is an open dense subset of the space $R(E,F)$ with respect to the
gap topology.
\end{theorem}
\begin{proof}Let $t \in R(E,F)$ and $F_{t}$ be its bounded
transform. We set $P_{n}= \frac{n}{n+1} \, F_{t}$, for all $n \in
\mathbb{N}$. Then for every $n \geq 1$ the operator $P_{n}$ is
bounded and satisfies $\|P_{n}\| < 1$. The operator
$1-F_{t}^*F_{t}$ has dense range and so is $1-P_{n}^*P_{n}$ (cf.
Remark 3.3). By Theorem 10.4 of \cite{LAN}, for any natural number
$n$ there exists a regular operator $T_{n}$ such that $
F_{T_{n}}=P_{n}=\frac{n}{n+1} \, F_{t}$. Therefore $
\|F_{T_{n}}\|=\|P_{n}\|<1$, so $T_{n}$ will be in $B(E,F)$. We
also have $ \sigma(T_{n},t)= \|F_{T_{n}}-F_{t}\|=\| \frac{n}{n+1}
\,F_{t}-F_{t} \| \rightarrow 0$. Recall that the Riesz topology
was stronger than the gap topology, that is $d(T_{n},t)
\rightarrow 0$.  $B(E,F)$ is therefore dense in $R(E,F)$ with
respect to the gap topology. Openness of $B(E,F)$ was given in
Lemma 3.1.
\end{proof}
\begin{corollary}The uniform structures
induced by the gap metric and by the operator norm on the space
of bounded adjointable operators are different. This follows from
the fact that the metric which is given by the usual norm of
bounded operator is complete while the gap metric on the set of
bounded adjointable operators is not complete.
\end{corollary}

\begin{lemma}Let $E$ be a Hilbert $ \mathcal{A}$-module and $t \in R(E)$.
Then $ \hat{t}: Dom( \,\hat{t} \,)=Dom(t) \times Dom(t^*)
\subseteq E \oplus E \longrightarrow E \oplus E, \ \hat{t}=
\left(\begin{array}{cc}0 & t^* \\t & 0
\end{array}\right) $ is a selfadjoint regular operator on the
Hilbert $ \mathcal{A}$-module $E \oplus E$.
\end{lemma}
\begin{proof}$t$ and $t^*$ are densely defined closed
operators and so is $ \hat{t}$. For each $(x,y),(u,v) \in Dom(\,
\hat{t}\,)$ we have
\begin{eqnarray*}
\langle \hat{t} \,(x,y),(u,v) \rangle=\langle (t^*y,tx),(u,v)
\rangle = \langle t^*y,u \rangle + \langle tx,v \rangle
&=& \langle y,t^{**}u \rangle+ \langle x,t^*v \rangle   \\
&=& \langle x,t^*v \rangle + \langle y,tu \rangle       \\
&=& \langle (x,y),(t^*v,tu) \rangle \,.
\end{eqnarray*}
Consequently, $Dom(\, \hat{t}\,)= Dom(\, \hat{t}^{*}\,)$ and $
\hat{t}^{*}= \hat{t}$. The operator $t$ is regular and so is
$t^*$, therefore the range of the operator
$$1+ \hat{t}^{*} \hat{t}=1+\hat{t}
^{2}=\left(\begin{array}{cc}1+t^*t & 0
\\0 & 1+tt^* \end{array}\right) $$ is dense in $ E \oplus E$.
That is, $ \hat{t}$ is a selfadjoint regular operator on $ E
\oplus E$.
\end{proof}

Let $E$ be a Hilbert $ \mathcal{A}$-module, then the operators of
the sets
$$SR(E \oplus E)_{odd}:= \{\left(\begin{array}{cc}0 & t^* \\t & 0
\end{array}\right)~:~ t \in R(E) \},$$
$$SB(E \oplus E)_{odd}:= \{\left(\begin{array}{cc}0 & T^* \\T & 0
\end{array}\right)~:~ T \in B(E) \},$$
are called {\it odd unbounded regular} and {\it odd bounded
adjointable} operators on the $ \mathbb{Z}/2$-graded Hilbert $
\mathcal{A}$-module $E \oplus E$, respectively. Odd operators
appear in Kasparov's KK-theory.

\begin{proposition}\label{gap-odd1}The map which associates to a regular operator
$t \in R(E)$ the selfadjoint operator $\hat{t}=
\left(\begin{array}{cc}0 & t^* \\t & 0 \end{array}\right)$ is an
isometric map from $R(E)$ onto $SR(E \oplus E)_{odd}$ with respect
to the gap metric, i.e. $d(t,s)=d( \, \hat{t}, \hat{s} \,)$, for
all $t,s \in R(E)$.
\end{proposition}
\begin{proof}Clearly the map $t \rightarrow \hat{t}$ is a
bijection from $R(E)$ onto $SR(E \oplus E)_{odd}$. So it is enough
to check that the map preserves the gap distance. For this end we
have
\begin{eqnarray*}
1+\hat{t}^{*} \hat{t}=1+\hat{t} ^{ \,
2}&=&\left(\begin{array}{cc}1+t^*t & 0
\\0 & 1+tt^* \end{array}\right),\\
(1+\hat{t}^{ \, 2})^{-1}&=&\left(\begin{array}{cc}(1+t^*t)^{-1} &
0 \\0 & (1+tt^*)^{-1} \end{array}~\right),\\
\hat{t}(1+\hat{t}^{ \, 2})^{-1}&=&\left(\begin{array}{cc}0 &
t^*(1+tt^*)^{-1}\\t(1+t^*t)^{-1} & 0 \end{array}\right).
\end{eqnarray*}
Therefor we have $ \|R_{ \hat{t}}-R_{ \hat{s}} \|={\rm sup} \{ \,
\|R_{t}-R_{s} \|,\|R_{t^*}-R_{s^*}\| \, \}$ and $ \|tR_{t}-sR_{s}
\| = \|\hat{t}R_{\hat{t}}-\hat{s}R_{\hat{s}} \|$, for $t,s \in
R(E)$. We get
\begin{eqnarray*}
d(t,s)&=& {\rm sup} \{ \, \|R_{t}-R_{s} \|,\|R_{t^*}-R_{s^*}
\|,\|tR_{t}-sR_{s} \|  \, \}\\
      &=& {\rm sup} \{ \, \|R_{ \hat{t}}-R_{ \hat{s}} \|
      \,, \|\hat{t}R_{\hat{t}}-\hat{s}R_{\hat{s}} \| \, \}
       = d( \hat{t}, \hat{s}).
\end{eqnarray*}
\end{proof}
\begin{corollary}\label{gap-odd2}Let $E$ be a Hilbert C*-module, then $SB(E
\oplus E)_{odd}$ is dense in $SR(E \oplus E)_{odd}$ with respect
to the gap topology.
\end{corollary}

By an arbitrary C*-algebra of compact operators $\mathcal{A}$ we
mean that $\mathcal{A}$=$c_{0}$-$ \oplus_{i \in I}\mathcal{K}
(H_{i})$, i.e. $\mathcal{A}$ is a $c_{0}$-direct sum of elementary
C*-algebras $\mathcal{K}(H_{i})$ of all compact operators acting
on Hilbert spaces $H_{i}, \ i \in I$ cf.~\cite[Theorem
1.4.5]{ARV}. Hilbert C*-modules over C*-algebras of compact
operators are generally neither self-dual nor countably
generated. However, they share many of their properties with
Hilbert spaces. Generic properties of these Hilbert C*-modules
have been studied by several authors, e.g. \cite{ARA, B-G, F-S,
FS2, GUL, N-S, PAL1, SCH}.    Let $\mathcal{A}$ be a C*-algebra,
then $\mathcal{A}$ is an arbitrary C*-algebra of compact
operators if and only if for every pair of Hilbert $
\mathcal{A}$-modules $E, F$, every densely defined closed
operator $t: Dom(t)\subseteq E \to F$ is regular (see \cite{F-S}).


The following results are borrowed from \cite{B-G, GUL}. Let
$\mathcal{A}$=$c_{0}$-$ \oplus_{i \in I}\mathcal{K} (H_{i})$ be a
C*-algebra of compact operators and $E$ be a Hilbert $
\mathcal{A}$-module. For each $i \in I$ consider the associated
submodule $E_i=\overline{ span \{ \mathcal{K} (H_{i})E \} }$.
Obviously, $\{E_i\}$ is a family of pairwise orthogonal closed
submodule of $E$ and it is well known (cf. \cite{SCH}) that $E$
admits a decomposition into the direct orthogonal sum $E=
\oplus_{i \in I} E_i$ as well as $E \oplus E=\oplus_{i \in I}
(E_i \oplus E_i)$. Suppose $t$ is a densely defined closed
operator on $E$ and $t_i:=t\, |_{Dom(t)\, \cap \, E_i}$, then
$G(t)=\oplus_{i \in I} G(t_i)$. This enable us to reduce our
attention to the case of a Hilbert C*-module over an elementary
C*-algebra $\mathcal{K}(H)$. Let $e \in \mathcal{K}(H)$ be an
arbitrary minimal projection and $E$ be a $
\mathcal{K}(H)$-module. Suppose $E_{e}:=e E=\{ex: \ x\in E \}$,
then $E_{e}$ is a Hilbert space with respect to the inner product
$(.,.)=trace\,(\langle.,.\rangle)$, which is introduced in
\cite{B-G}. Let $B(E)$ and $B(E_{e})$ be C*-algebras of all
bounded adjointable operators on Hilbert $ \mathcal{K}(H)$-module
$E$ and Hilbert space $E_{e}$, respectively. Baki\'c and
Gulja\v{s} have shown that the map $ \Psi : B(E) \to B(E_{e}) , \
\Psi(T)= T| _{E_{e}} $ is a $*$-isomorphism of C*-algebras
\cite[Theorem 5]{B-G}.

suppose $R(E)$ and $R(E_{e})$ are the spaces of densely defined
closed operators on Hilbert $ \mathcal{K}(H)$-module $E$ and
Hilbert space $E_{e}$, respectively. Then $ \psi : R(E) \to
R(E_{e}) , \ \psi(t)= t|_{E_{e}} $ is a bijection operation
preserving map of $R(E)$ onto $R(E_{e})$, cf. \cite[Theorem
1]{GUL}. By the restriction $t|_{E_{e}}$ of an operator $t \in
R(E)$ we mean the restriction of $t$ onto the subspace $eDom(t)$,
where $eDom(t) \subseteq E_{e}$ and $ \overline{eDom(t)}=E_{e}$.

\begin{theorem}\label{gap-isometric}Let $E$ be a Hilbert $\mathcal{K}(H)$-module and
$e \in \mathcal{K}(H)$ be any minimal projection. We equip $R(E)$
and $R(E_{e})$ with the gap metric, then $ \psi : R(E) \to
R(E_{e}) , \ \psi(t)= t_{e} $ is an isometric operation
preserving map of $R(E)$ onto $R(E_{e})$
\end{theorem}
\begin{proof} $ \psi$ is a bijection operation preserving
map of $R(E)$ onto $R(E_{e})$ by \cite[Theorem 1]{GUL}. Let $t$
be a regular operator on $E$ and $t_e:=t|_{E_{e}}$, then
$(R_t)|_{E_{e}}$ and $(tR_{t})|_{E_{e}}$ are bounded operators on
the Hilbert space $E_{e}$ and $(R_{t}) |_{E_{e}}= R_{t _{e}}$ and
$(tR_{t})|_{E_{e}}= t_{e} \,R_{t _{e}}$. For $t,s \in R(E)$, the
equalities $\|R_{t}-R_{s} \| =\|R_{t _{e}}-R_{s _{e}}\| \,, ~
\|R_{t^*}-R_{s^*} \|=\|R_{t_{e}^*}-R_{s_{e}^*}\| \,, ~
\|tR_{t}-sR_{s} \|= \| t_{e} R_{t _{e}}- s_{e} R_{s_{e}}\| \,$
hold by utilizing \cite[Theorem 5]{B-G}. Therefore
$d(t,s)=d(\psi(t),\psi(s))$ as we required.
\end{proof}
The above theorem lifts the properties of the gap metric from the
space of densely defined closed operators on Hilbert spaces to
the space of densely defined closed operators on Hilbert
C*-modules over C*-algebras of compact operators.\\\

\section{Connectivity}
Unbounded Fredholm operators has been studied systematically in
the papers \cite{BLP, C-L, JOA} and the book \cite{KAT}. In this
section we use Theorem \ref{gap-isometric} to classify the
path-components of the set of regular Fredholm operators in
Hilbert $\mathcal{K}(H)$-modules.

Suppose $E$ is a Hilbert $ \mathcal{A}$-module. Recall that a
bounded operator $T\in B(E)$ is said to be {\it Fredholm} (or $
\mathcal{A}$-Fredholm) if there exists $G\in B(E)$ such that
$GT=TG=1 \ \ mod \ K(E)$.  Consider a regular operator $t$ on $E$.
An adjointable bounded operator $G\in B(E)$ is called a {\it
pseudo left inverse} of t if $Gt$ is closable and its closure $
\overline{Gt}$ satisfies $ \overline{Gt}\in B(E)$ and $
\overline{Gt}=1 \ \ mod \ K(E)$. Analogously $G$ is called a {\it
pseudo right inverse} if $tG$ is closable and its closure $
\overline{tG}$ satisfies $ \overline{tG}\in B(E)$ and $
\overline{tG}=1 \ \ mod \ K(E)$. The regular operator t is called
{\it Fredholm} (or $ \mathcal{A}$-Fredholm), if it has a pseudo
left as well as a pseudo right inverse. The regular operator $t$
is Fredholm if and only if $F_{t}$ is, cf. \cite{JOA, N-S}. For
the general theory of bounded and unbounded Fredholm operators on
Hilbert C*-modules we refer to \cite{GVF, JOA, WEG}. Let
$FredR(E)$ denote the space of regular Fredholm operators on $E$,
equipped with the gap topology, and let $FredSR(E)$ denote the
subspace consisting of the set selfadjoint regular Fredholm
operators.

Every two operators in $FredR(E)$ are called {\it homotopic} if
they are in the same path-connected component of $FredR(E)$.
It is natural to ask for a characterization of the path-components
of $FredR(E)$. This question was completely answered by Cordes
and Labrousse in \cite{C-L} in the case of Hilbert spaces.
\begin{theorem}\label{Cordes-Labrousse}Let $H$ be a Hilbert
space of infinite dimension. Every two operators in $FredR(H)$
have the same index if and only if they are homotopic, i.e. they
can be connected by a continuous path in $FredR(H)$.
\end{theorem}

Suppose $E$ is a Hilbert $ \mathcal{K}(H)$-module and $t$ is a
regular operator on $E$. Then $t$ is Fredholm if and only if the
range of $t$ is a closed submodule of $E$ and both $dim_{
\mathcal{K}} Ker (t)$, $dim_{ \mathcal{K}} Ker (t^*)$ are finite.
In this case we can define an index of $t$ formally, i.e. we can
define:
$$ ind \,t=dim_{ \mathcal{K}} Ker (t) - dim_{ \mathcal{K}} Ker ( t^*).$$
We refer to the publications
 \cite{B-G, MIN, N-S} for the proof of the
preceding results. More information about orthonormal Hilbert
bases can be found in \cite{ARA, F-L}.  Apply Theorems
\ref{gap-isometric}, \ref{Cordes-Labrousse} to get the following
fact.

\begin{corollary}\label{sharifi-compact-component}Let $E$ is a Hilbert
$\mathcal{K}(H)$-module. Every two operators in $FredR(E)$ have
the same index if and only if they are homotopic, i.e. they can
be connected by a continuous path in $FredR(E)$.
\end{corollary}


\begin{corollary}\label{Sharifi-compact}Suppose $E$ is a Hilbert $\mathcal{K}(H)$-module.
The space $FredSR(E)$ is path-connected and the space $FredR(E)$
is not path-connected.
\end{corollary}
For the proof, just recall that any element of $FredSR(E)$ has
zero index and then apply Corollary
\ref{sharifi-compact-component}.

We close the paper with the notification that the previous
corollary may fail for some other C*-algebra of coefficients. To
find an example one can use a result due to Joachim \cite[Theorem
3.5]{JOA}. Indeed, Joachim's theorem is a remarkable
generalization of a result due to Atiyah and Singer
\cite{Ati-Singer} which describes the space of regular (resp.,
selfadjoint regular) Fredholm operators on standard Hilbert
C*-modules over unital C*-algebra of coefficients. Mingo also
gave a description of path-components in the set of bounded
Fredholm operators on standard Hilbert modules \cite{MIN}.

{\bf Acknowledgement}: This research was supported by a grant of
Shahrood University of Technology. It was done during author's
stay at the Mathematisches Institut of the Westf\"{a}lische
Wilhelms-Universitat\"{a}t M\"{u}nster, Germany, in 2007. He
would like to express his thanks to professor Michael Joachim and
his colleagues in the topology and operator algebras groups for
warm hospitality. The author is also grateful to Professor M.
Frank and the referee for their careful reading and their useful
comments.



\begin{thebibliography}{99}
\bibitem {ARA} Lj. Aramba\v{s}i\'c, Another characterization of Hilbert C*-modules
over compact operators, {\it J. Math. Anal. Appl.}  {\bf
344}(2008), no. 2, 735-740.
\bibitem {ARV} W. Arveson, {\it An Invitation to C*-algebras}, Springer, New York, 1976.
\bibitem {Ati-Singer} M. F. Atiyah and I. M. Singer,  Index theory for skew-adjoint Fredholm
operators, {\it IHES Publ. Math.} {\bf 37}(1969), 305-326
\bibitem {B-J} S. Baaj and P. Julg, Th\'{e}orie bivariante de
Kasparov et op\'{e}rateurs non born\'{e}s dans les C*-modules
hilbertiens, {\it C. R. Acad. Sci., Paris, Series I} {\bf
296}(1983), 875-878.
\bibitem {B-G} D. Baki\'c and B. Gulja\v{s}, Hilbert C*-modules
over C*-algebras of compact operators, {\it Acta Sci. Math.}
(Szeged) {\bf 68}(2002), no. 1-2, 249-269.
\bibitem {BLP} B. Boss-Bovenbek, M. Lesch and J. Phillips, Unbounded
Fredholm operators and spectral flow, \textit{Canada. J. Math.}
\textbf{57}(2005), no. 2, 225-250.
\bibitem {C-L} H. O. Cordes  and J.~P. Labrousse, The
invariance of the index in the metric space of closed operators,
{\it J. Math. Mech.} {\bf12}(1963), 693--719.

\bibitem {FR2} M. Frank, Geometrical aspects of Hilbert C*-modules,
{\it Positivity} {\bf 3}(1999), 215-243.
\bibitem {FR3} M. Frank, Self-duality and C*-reflexivity of
Hilbert C*-modules, {\it Zeitschr. Anal. Anwendungen} {\bf
9}(1990), 165-176.
\bibitem{F-L}  M. Frank, D. R. Larson, Frames in Hilbert C*-modules and
  C*-algebras, {\it J.~Operator Theory} {\bf 48}(2002), 273-314.
\bibitem {F-S} M. Frank and K. Sharifi, Adjointability of densely
defined closed operators and the Magajna-Schweizer Theorem, to
appear in {\it J. Operator Theory}.
\bibitem {FS2} M. Frank and K. Sharifi, Generalized inverses and polar
decomposition of unbounded regular operators on Hilbert
C*-modules, to appear in {\it J. Operator Theory}.
\bibitem {GUL} B. Gulja\v{s}, Unbounded operators on Hilbert
C*-modules over C*-algebras of compact operators, {\it J. Operato
Theory}  {\bf 59}(2008),  no. 1, 179-192.
\bibitem {GVF} J. M. Gracia-Bond\'{\i}a and J. C. V\'arilly and H. Figueroa,
{\it Elements of non-commutative geometry}, Birkh\"auser, 2000.
\bibitem {JOA} M. Joachim, Unbounded Fredholm operators and
K-theory, \textit{High-dimensional manifold topology}, 177-199,
\textit{World Sci. Publishing, River Edge, NJ}, (2003).
\bibitem {KAT} T. Kato, {\it Perturbation theory for linear operators},
Springer Verlag, New York, 1984.
\bibitem {KAU} W. E. Kaufman, A stronger metric for closed operators in Hilbert space,
{\it Proc. Amer. Math. Soc.} {\bf 90}(1984), no. 1, 83-87.
\bibitem {KK} D. Kucerovsky, The $KK$-product of unbounded modules,
{\it K-Theory} {\bf 11}(1997), 17-34.
\bibitem {LAN} E. C. Lance, {\it Hilbert C*-Modules}, LMS Lecture Note Series 210,
Cambridge Univ. Press, 1995.
\bibitem {MIN} J. A. Mingo, K-theory and multipliers of stable
C*-algebras, {\it Trans. Amer. Math. Soc.} {\bf 299}(1987),
397-411.
\bibitem {NIC} L. Nicolaescu, On the space of Fredholm operators,
available on arXiv:math.DG/0005089 v1 9 May 2000.
\bibitem {N-S}  A. Niknam and K. Sharifi, The Atkinson Theorem in
Hilbert C*-modules over C*-algebras of compact operators, {\it
Abstr. Appl. Anal.}  (2007), Art. ID 53060.
\bibitem {PAL1} A. Pal, Regular operators on Hilbert $C^*$-modules,
{\it J. Operator Theory} {\bf 42}(1999), no. 2, 331-350.
\bibitem {SCH} J. Schweizer, A description of Hilbert C*-modules
in which all closed submodules are orthogonally closed, {\it Proc.
Amer. Math. Soc.} {\bf 127}(1999), 2123-2125.

\bibitem {WEG} N. E. Wegge-Olsen, {\it K-theory and C*-algebras: a Friendly Approach},
Oxford University Press, Oxford, England, 1993.
\bibitem {W-N} S. L. Woronowicz and K. Napi\'{o}rkowski, Operator
theory in the C*-algebra framework, {\it Rep. Math. Phys.} {\bf
31}(1992) 353-371.
\bibitem {WOR} S. L. Woronowicz, Unbounded elements affiliated with C*-algebras and
noncompact quantum groups, {\it Comm. Math. Phys.} {\bf
136}(1991), 399-432.
\end{thebibliography}
\end{document}